\documentclass[12pt,oneside]{amsproc}

\usepackage[cp1251]{inputenc}
\usepackage[russian,english]{babel}
\usepackage{amssymb,epic,eepic}
\usepackage[dvips]{graphicx}
\numberwithin{equation}{section}
\newtheorem{lem}{Lemma}[section]
\newtheorem{thm}{Theorem}[section]

\newtheorem{rim}{Remark}[section]

\newcommand{\Wo}{{\raisebox{0.2ex}{$\stackrel{\circ}{W}$}}{}}

\title[]{On Sharp Estimates of Derivatives of Even Order\\
in Sobolev Spaces}
\author{T.~A.~Garmanova and I.~A.~Sheipak}
\address{Faculty of Mechanics and Mathematics, Lomonosov Moscow State University,
Moscow, 119991 Russia}
\email{garmanovata@gmail.com, iasheip@yandex.ru}
\thanks{This work was supported 
by the Russian Foundation for Basic Research
under grant~19-01-00240 (Secs. 2--4) and
by the Russian Science Foundation
under grant~???? (Secs. 5--6).}

\begin{document}

\noindent UDC 517.984, 517.518.23

\begin{abstract}
 The norms of embedding operators of Sobolev spaces
$\Wo^n_2[0;1]\hookrightarrow\Wo^k_\infty[0;1]$
($0\leqslant k\leqslant n-1$)
are considered.
 The least possible quantities
$A^2_{n,k}(x)$
in the inequalities
$|f^{(k)}(x)|^2\leqslant
A^2_{n,k}(x)\|f^{(n)}\|^2_{L_2[0;1]}$
are studied.
 On the basis of the relations
between the
$A^2_{n,k}(x)$
and the antiderivatives of the Legendre polynomials,
the properties of the maxima of the functions
$A^2_{n,k}(x)$
are established.
 It is shown
that, for all~$k$,
the global maximum of the function
$A^2_{n,k}$
on the closed interval
$[0;1]$
is the maximum point nearest to the midpoint of the interval;
in particular,
for even~$k$,
$x=1/2$ is such a point.
 For the parameter~$k$ of even order,
an explicit formula for the norms of the embedding operators
is obtained.
\end{abstract}

\maketitle


\section{Introduction}
 We consider the problem of finding the norms of embedding operators of the Sobolev space~$\Wo_2^n[0;1]$
in the space~$\Wo_\infty^k[0;1]$.
 By
$\Wo_p^n[0;1]$
($n\in\mathbb{N}$)
we mean the space consisting of functions~$f$
all of whose derivatives up to order $n-1$
are absolutely continuous
on the closed interval
$[0;1]$, $f^{(n)}\in L_p[0;1]$,
and the following boundary conditions hold:
$$
f^{(j)}(0)=f^{(j)}(1)=0,\quad j=0,1,\ldots, n-1.
$$

 By
$\mathcal H$
we shall denote the space~$\Wo_2^n[0;1]$
with norm
$$
\|f\|_{\mathcal H}:=\left(\int_0^1\left|f^{(n)}(x)\right|^2dx\right)^{1/2}.
$$

 For a fixed point
$x\in(0;1)$
and an integer
$k\in\{0,1,\ldots, n-1\}$,
we consider the problem of finding the norm of the functional
$f\mapsto f^{(k)}(x)$
in the space~$\mathcal H$:
\begin{equation}
\label{eq:ekstr}
A_{n,k}(x):=\sup\left\{|f^{(k)}(x)|:\quad \|f\|_{\mathcal H}=1\right\}.
\end{equation}

 The exact embedding constant is defined
by the formula
\begin{equation}
\label{eq:const_emb}
\Lambda_{n,k}:=\sup_{x\in[0;1]} A_{n,k}(x)
\end{equation}
and is equal to the norm of the embedding operator
$\mathcal H\to \Wo^k_\infty[0;1]$.
 It is more convenient to calculate the quantities
$A^2_{n,k}(x)$.
 This is related to the fact that the quantities
$A^2_{n,k}(x)$
satisfy recurrence relations
in which the antiderivatives of the Legendre polynomials appear.

 Estimates of derivatives of intermediate order
in terms of the norms of Sobolev spaces
arise in many problems.
 In addition,
studies of embedding constants often involve
elegant relations with special functions,
in particular, with Legendre polynomials
(see~\cite{Kalyab}) and Bessel functions
(see~\cite{Naz});
thus, studies of embedding constants are closely related
to function theory.
 Certain questions in the embedding theory for Sobolev spaces
are of interest in connection with the spectral theory of operators,
for example, these are closely connected
with the problem of finding smallest eigenvalues for certain differential operators
(see~\cite{Kalyab} and the references therein).
 The problem of finding exact embedding constants
has a rather long history.
 In the present paper,
we are not concerned with questions
related to the embedding of the Sobolev spaces
$$
\Wo_p^n[0;1] \hookrightarrow \Wo_q^k[0;1]
$$
for arbitrary indices~$p$ and~$q$.
 A short survey of the history of such problems
was given in~\cite{NazM}.

 In~\cite{Kalyab}, explicit expressions for
$A^2_{n,0}(x)$, $A^2_{n,1}(x)$,
and
$A^2_{n,2}(x)$
for
$x\in[-1;1]$ were obtained,
their global maxima were found,
$\Lambda^2_{n,0}$, $\Lambda^2_{n,1}$,
and
$\Lambda^2_{n,2}$ were calculated,
and it was conjectured that, for even~$k$,
the maximum of
$A^2_{n,k}$
is attained
at the midpoint of a closed interval
and,
for odd~$k$
$A^2_{n,k}$ has a minimum
at the midpoint of a closed interval.

 In~\cite{NazM},
it was proved that
the midpoint of a closed interval
is the point at which the functions
$A^2_{n,k}$
have a local maximum for even~$k$
and a local minimum
for odd~$k$.
 In addition,
in~\cite{NazM}, the constants
$\Lambda^2_{n,4}$
and
$\Lambda^2_{n,6}$
were calculated,
but, because of an arithmetic error
at the end of the calculations,
false values of these quantities were obtained.
 It was also shown in~\cite{NazM} that, for
$k=2,4$
and
$n>k$,
the extremal function
possesses symmetry with respect to the midpoint of a closed interval.
 We note that, for
large~$k$ and~$n$,
the difficulties involved in the calculation of the quantities
$A^2_{n,k}$
increase significantly.
 At the same time,
as is seen, for example, from results given in~\cite{Kalyab}
and~\cite{GarmSh},
the cases of odd values of~$k$
are considerably more complicated than those of even values.

 The approach in~\cite{Kalyab}
was further developed in~\cite{GarmSh},
where explicit formulas for the functions
$A^2_{n,3}$, $A^2_{n,4}$, $A^2_{n,5}$
were obtained
and the exact constants
$\Lambda_{n,3}$
and
$\Lambda_{n,5}$
were found.
 Also,
in that paper,
in finding the global maximum of the quantity
$A^2_{n,5}(x)$, $x\in(0;1)$,
the authors had to calculate the extrema of some very
awkward expressions.
 As a result, it becomes more difficult
to determine which local maximum point of the function
$A^2_{n,k}(x)$
is a global maximum point.
 In~\cite{GarmSh}, the determination of the global maximum of the function
$A^2_{n,5}$
involved elaborate technical manipulations.
 This leads to the conclusion
that using former methods, one can hardly find the global maximum of
$A^2_{n,k}(x)$
for
arbitrary
$k\geqslant 6$.
 However, the previous studies provided help
for understanding the structures of the derivatives of the functions
$A^2_{n,k}(x)$
for arbitrary~$n$ and~$k$
(see Lemma~\ref{lem:APP}).

 It is more convenient
to consider embedding constants
on the closed interval
$[0;1]$.
 This is related to the fact that the functions
$A^2_{n,k}$
are recursively related to the antiderivatives of the Legendre polynomials
and, for them, the formulas
on the closed interval
$[0;1]$
are less cumbersome.

 In connection with the foregoing,
the question of finding the global maximum points of the functions
$A^2_{n,k}(x)$, $x\in(0;1)$
for all
$n\in \mathbb{N}$
and
$0\leqslant k\leqslant n-1$
arises.
 It was conjectured in~\cite{GarmSh} that a global maximum
is attained at the maximum point
nearest to the midpoint of the closed interval under consideration.
 The studies carried out in~\cite{Kalyab},~\cite{NazM},
and, especially, in~\cite{GarmSh} show that
it is simpler to consider other variables in the embedding constants.
 On the closed interval
$[-1;1]$, it is better to pass to $t=x^2-1$, while,
in the case
$x\in[0;1]$,
it is better to make the change of variable
$t:=x^2-x\in[-\frac14;0]$.
 In these variables,
the conjecture on the global maximum
is as follows:
\textit{for arbitrary~$k$,
the global maximum of
$A^2_{n,k}(t)$
is attained
at the maximum point nearest to
$-1/4$\textup;
in particular,
for even~$k$,
the point
$t=-1/4$
is a global maximum point.}
 The goal of the present paper
is, first, to prove this conjecture
and, second, to obtain the explicit form of the exact embedding constants
$\Lambda^2_{n,k}$
for all even~$k$.

 The structure of the paper is as follows.
 In Sec.~2,
we present some necessary recurrence relations
for the Legendre polynomials
on the closed interval
$[0;1]$;
in Sec.~3,
we prove the relations between the functions
$A^2_{n,k}(x)$
and the antiderivatives of the Legendre polynomials;
in Sec.~4,
we prove the properties of the maxima of the function
$A^2_{n,k}$
and establish the relationship between the extremal points
and the zeros of the corresponding antiderivatives of the Legendre polynomials;
in Sec.~5,
we find the explicit form of the exact embedding constants
on the closed interval
$[-1;1]$
for
arbitrary
$n\in\mathbb{N}$
and all even
$k\in[0;n-1]$;
in Sec.~6,
we point out the relationship between the embedding constants
and the smallest eigenvalue from the class of spectral problems
with coefficients-distributions
as well as give the recalculation formulas
for the functions
$A^2_{n,k}$
determined
on the closed intervals
$[0;1]$
and
$[-1;1]$,
respectively.

 Throughout the paper, except for Sec.~5,
by
$P_n$
we denote shifted Legendre polynomials
generating an orthogonal basis
in the space~$L_2[0;1]$.
 In Sec.~5, this notation
is used for the Legendre polynomials on the closed interval $[-1;1]$.

 By
$t$
we denote the change
$t:=x^2-x$, $x\in[0;1]$, $t\in\left[-\frac14;0\right]$.
 By an \textit{exact} embedding constant
we mean the quantity
$\Lambda_{n,k}$.

\section{Recurrence relations for the antiderivatives of the Legendre polynomials.}

 The Legendre polynomials
$P_n$, $n=0,1,\ldots$
on the closed interval
$[0;1]$
are given by the Rodrigues formula
\begin{equation}
\label{eq:Leg}
P_n(x):=\frac{1}{n!}\left((x^2-x)^n\right)^{(n)}.
\end{equation}
  The antiderivative of order $m\geqslant 0$
is defined by the formula
\begin{equation}
\label{eq:LegPrim}
P_n^{(-m)}(x):=\frac{1}{n!}\left((x^2-x)^n\right)^{(n-m)}.
\end{equation}

 As a rule,
we will work with Legendre polynomials
and their antiderivatives
using the variable
$t:=x^2-x$.
 At the same time,
according to formulas~\eqref{eq:Leg},~\eqref{eq:LegPrim},
the derivative will be regarded
as a derivative with respect to the variable~$x$.

\begin{lem}
 The antiderivatives of shifted Legendre polynomials
satisfy the relation
$$
P_{n+1}^{(k-n+1)}=2(k+1)P_{n}^{(k-n)}+(2x-1) P_{n}^{(k-n+1)}.
$$
\end{lem}

\begin{proof}
 Since
$t=x^2-x$
and
$\dfrac{d}{dx}=(2x-1)\dfrac{d}{dt}$,
by the definition of the antiderivative of the Legendre polynomial,
we have
\begin{equation}
\label{eq:Pder1}
P_{n}^{(k-n+1)}=\left(\dfrac{t^{n+1}}{(n+1)!}\right)^{(k+2)}=\left(\dfrac{t^{n}}{n!}(2x-1)
\right)^{(k+1)}=
(2x-1)\left(\dfrac{t^{n}}{n!}\right)^{(k+1)}+2(k+1)\left(\dfrac{t^{n}}{n!}\right)^{(k)}.
\end{equation}
 On the other hand,
\begin{equation}
\label{eq:Pder2}
P_{n}^{(k-n)}=\left(\dfrac{t^n}{n!}\right)^{(k)},\qquad
P_{n}^{(k-n+1)}=\left(\dfrac{t^{n}}{(n)!}\right)^{(k+1)}.
\end{equation}

 The assertion of the lemma follows from
relations~\eqref{eq:Pder1},~\eqref{eq:Pder2}.
\end{proof}

\begin{lem}
\label{lem:Pdifur}
 The function
$P_n^{(-m)}$,
the antiderivative of order $m\geqslant 0$
of the shifted Legendre polynomial~$P_n$,
satisfies the differential equation
\begin{equation}
\label{eq:Pdifur}
(x-x^2)y''+(2x-1)(m-1)y'+(n+m)(n-m+1)y=0.
\end{equation}
\end{lem}

\begin{proof}
1)
 For
$m=0$,
the shifted Legendre polynomial~$P_n$
satisfies the equation
$$
(x-x^2)y''-(2x-1)y'+n(n+1)y=0.
$$

2)
 For the function
$P_n^{(-m)}$
($m\geqslant 1$),
we shall search for the corresponding differential equation
in the form
\begin{equation}
\label{eq:Pdifur_gen}
(x-x^2)y''+c_{n,m}(2x-1)y'+d_{n,m}y=0
\end{equation}
and find the recurrence relations
for the coefficients
$c_{n,m}$
and
$d_{n,m}$.
 It follows from Sec.~1 that
$c_{n,0}=-1$, $d_{n,0}=n(n+1)$.

 Differentiating Eq.~\eqref{eq:Pdifur_gen},
we see that the polynomial
$(P_n^{(-m)})'=P_n^{(-(m-1))}$
satisfies the equation
$$
(x-x^2)y''+(2x-1)(c_{n,m}-1)y'+(d_{n,m}+2c_{n,m})y=0,
$$
whence
$c_{n,m}=c_{n,m-1}+1$, $d_{n,m}=d_{n,m-1}-2c_{n,m}$.
 Solving these recurrence relations,
we obtain
$$
c_{n,m}=m-1,\qquad d_{n,m}=(n+m)(n-m+1).
$$
\end{proof}

 Similar formulas for the antiderivatives of the classical Legendre polynomials
were obtained in~\cite{HolSh}.

\begin{lem}
 For the antiderivatives of shifted Legendre polynomials,
the following relations hold{\rm:}
\begin{equation}
\label{eq:Prec}
2(2n+1) P_{n}^{(-m)}=P_{n+1}^{(-m+1)}-P_{n-1}^{(-m+1)}.
\end{equation}
\end{lem}
\begin{proof}
 The derivatives are regarded
as derivatives with respect to the variable~$x$.
 Since $t=x^2-x$
and
$\dfrac{d}{dx}=(2x-1)\dfrac{d}{dt}$,
by the definition of the antiderivative of the Legendre polynomial,
we have
$$
P_{n+1}^{(-m+1)}=\left(\dfrac{t^{n+1}}{(n+1)!}\right)^{(n+2-m)}=\dfrac{1}{n!}(t^n(2x-1))^{
(n+1-m)}=
\dfrac{1}{n!}(nt^{n-1}(4t+1)+2t^n)^{(n-m)}.
$$
 After simplifications, we see that
\begin{equation}
\label{eq:Pprim1}
P_{n+1}^{(-m+1)}=2(2n+1)\left(\dfrac{t^n}{n!}\right)^{(n-m)}+
\left(\dfrac{t^{n-1}}{(n-1)!}\right)^{(n-m)}.
\end{equation}
 Similarly, we obtain
\begin{equation}
\label{eq:Pprim2}
P_{n-1}^{(-m+1)}=\left(\dfrac{t^{n-1}}{(n-1)!}\right)^{(n-m)},\quad
P_{n}^{(-m)}=\left(\dfrac{t^{n}}{n!}\right)^{(n-m)}.
\end{equation}
 The assertion of the lemma follows from relations~\eqref{eq:Pprim1},~\eqref{eq:Pprim2}.
\end{proof}

 In particular,
for the antiderivative of order $m=n-k$
of the Legendre polynomials
used in the embedding constants,
we obtain
\begin{equation}
\label{eq:Pprimk}
2(2n+1) P_{n}^{(k-n)}= P_{n+1}^{(k-n+1)}- P_{n-1}^{(k-n+1)}.
\end{equation}

\section{Relations between the Legendre polynomials
and the functions $A^2_{n,k}$}

 Let us recall
that, in~\cite{GarmSh},
the following relation was obtained:
\begin{equation}
\label{eq:PA}
A^2_{n,k}(x)=A^2_{n-1,k-1}(x)-(2n-1)\left(P_{n-1}^{(k-n)}(x)\right)^2.
\end{equation}

\begin{lem}
\label{lem:APP}
 For the functions
$A^2_{n,k}(x)$,
the following relation holds{\rm:}
$$
\dfrac{d}{dx}(A^2_{n,k})=-P_{n-1}^{(k-n+1)}\cdot P_{n}^{(k-n+1)}.
$$
\end{lem}

\begin{proof}
 Let us carry out the proof by the method of mathematical induction.

1) Since $A^2_{n,1}(t)=\dfrac{-t^{2n-3}}{((n-1)!)^2
(2n-3)}\left[(2n-1)(2n-3)t+(n-1)^2\right]$
(see~\cite{GarmSh}),
after simplifications,
we obtain
$$
\dfrac{d}{dx}(A^2_{n,1})=\dfrac{-t^{2n-4}(2a-1)(n-1)}{((n-1)!)^2}\left[2(2n-1)t+n-1\right].
$$
 On the other hand,
\begin{align*}
P_{n-1}^{(2-n)}&=\dfrac{d}{dx}\left(\dfrac{t^{n-1}}{(n-1)!}\right)=\dfrac{t^{n-2}(2x-1)}{(n
-2)!},
\\
P_{n}^{(2-n)}&=\dfrac{d^2}{dx^2}\left(\dfrac{t^{n}}{n!}\right)=
\dfrac{t^{n-2}}{(n-1)!}\left[2(2n-1)t+n-1\right].
\end{align*}
 Multiplying the last two equalities,
we see that the assertion holds
for
$k=1$
and any~$n$.

2) Let the assertion hold
for some~$k$ and any~$n$.
 Since
$$
A^2_{n+1,k+1}(x)=A^2_{n,k}(x)-(2n+1)\left(P_n^{(k-n)}\right)^2,
$$
differentiating with respect to~$x$,
we obtain
$$
(A^2_{n+1,k+1})'=(A^2_{n,k})'-2(2n+1)P_n^{(k-n)}\cdot P_n^{(k-n+1)}.
$$
 Replacing $(A^2_{n,k})'$
by  the expression
$(A^2_{n,k})'=-P_{n-1}^{(k-n)}\cdot P_{n}^{(k-n)}$,
which is valid by the induction assumption,
and the multiplier
$2(2n+1) P_n^{(k-n)}$
given by formula~\eqref{eq:Pprimk}
by the difference
$P_{n+1}^{(k-n+1)}-P_{n-1}^{(k-n+1)}$,
we obtain the required relation.
\end{proof}

\begin{rim}
\label{rim}
 In view of Eq.~\eqref{eq:Pdifur},
it follows from Sturm theory that,
between any two zeros of the polynomial
$P_{n}^{(k-n+1)}$,
there is at least one zero of
$P_{n-1}^{(k-n+1)}$.
\end{rim}

 We shall need more subtle properties of the zeros of the polynomials
$P_{n-1}^{(k-n+1)}$
and
$P_{n}^{(k-n+1)}$
that do not follow from Sturm theory.
 Let us prove a more general result.

\begin{lem}
\label{lem:zeroes}
 The roots of the polynomials
$P_{n-1}^{(k-n+1)}$
and
$P_{n}^{(k-n+1)}$
lying
on the interval
$(0;1)$
alternate.
 The zeros of
$P_{n-1}^{(k-n+1)}$
are the minimum points of the function
$A^2_{n,k}(x)$
and
the zeros of
$P_{n}^{(k-n+1)}$
are its maximum points.
\end{lem}
\begin{proof}
 Let us carry out the proof by the method of mathematical induction

1)
 The induction base:
for any
$n\geqslant 1$
and
$k=0$,
the given assertion holds.
 Indeed,
in this case,
$$
(A^2_{n,0})'=-P_{n-1}^{(-n+1)}\cdot P_{n}^{(-n+1)}=-\dfrac{t^{n-1}}{(n-1)!}\cdot
\dfrac{t^{n-1}}{(n-1)!}(2x-1).
$$
$P_{n-1}^{(-n+1)}$
has no zeros inside the interval
$(0;1)$,
while
$P_{n}^{(-n+1)}$
has exactly one zero
$x=1/2$,
and this is a maximum point.

2) Let,
for
some
$n$
and all
$k\leqslant n-1$,
the assertion hold
for the pair of polynomials
$P_{n-1}^{(k-n+1)}$
and
$P_{n}^{(k-n+1)}$;
let us prove it
for
$P_{n}^{(k-n+1)}$
and
$P_{n+1}^{(k-n+1)}$
(i.e., we pass to the pair of indices
$n+1$
and
$k+1$).
 By Lemma~\eqref{lem:APP} and formula~\eqref{eq:Pprimk},
we obtain
\begin{align*}
-(A^2_{n+1,k+1})'&= P_{n}^{(k-n+1)} P_{n+1}^{(k-n+1)}= -(A^2_{n,k})'+2(2n+1) P_{n}^{(k-n)}
P_{n}^{(k-n+1)}\\
&= P_{n-1}^{(k-n+1)} P_{n}^{(k-n+1)} + 2(2n+1) P_{n}^{(k-n)} P_{n}^{(k-n+1)}.
\end{align*}
 Since $P_{n}^{(k-n+1)}$
is the derivative of $P_{n}^{(k-n)}$,
it follows,
by Rolle's theorem,
that their zeros alternate and, further,
the zeros of
$P_{n}^{(k-n+1)}$
are the maximum points of the function
$(P_{n}^{(k-n)})^2$
and the zeros of
$P_{n}^{(k-n)}$
are its minimum points.

 By the induction hypothesis,
the zeros of
$P_{n-1}^{(k-n+1)}$
and
$P_{n}^{(k-n+1)}$
alternate.
 The point
$x_0 = {1}/{2}$
is
a zero of
$P_{n}^{(k-n+1)}$
if and only if
$k$
is even,
and
it is known from~\cite{NazM} that the point
$x_0 = {1}/{2}$
is a local maximum of
$A_{n,k}$
for even~$k$
and a local minimum
for odd~$k$.
 Thus, the zeros of
$P_{n}^{(k-n+1)}$
are the maximum points of
$A_{n,k}$.

 Then
it follows from formula~\eqref{eq:PA}
that the zeros of
$P_{n}^{(k-n+1)}$
are the minimum points of
$A^2_{n+1,k+1}$
and the zeros of
$P_{n+1}^{(k-n+1)}$
are the maximum points of
$A^2_{n+1,k+1}$,
respectively,
which implies the assertion of the lemma.
\end{proof}

\section{The global maximum of $A^2_{n,k}$}
\label{sub:Anglob}

\begin{thm}
\label{thm:derAPP}
 The values of the quantities
$A^2_{n,k}$
at their local maximum points
lying
on the closed interval
$[0,{1}/{2}]$
constitute a nondecreasing sequence
and those
at the local maximum points
lying
on the closed interval
$[\frac{1}{2},1]$
constitute a nonincreasing sequence
(see the figure).
\end{thm}

\begin{proof}
 For each fixed
$n$
and
$k$,
we consider the function
$$
B_{n,k}(x) = A^2_{n,k}(x) + (f(x)+g(x)) (P_{n}^{(k-n+1)}(x))^2
$$
with
some,
so far arbitrary, differentiable functions~$f$ and~$g$.
 The concrete form of these functions will be indicated in what follows.
 Then
\begin{align*}
B'_{n,k}(x) &= (A^2_{n,k})' + (f'(x)+g'(x)) (P_{n}^{(k-n+1)})^2 + 2 (f(x)+g(x))
P_{n}^{(k-n+1)} P_{n}^{(k-n+2)}\\
&=-P_{n-1}^{(k-n+1)} P_{n}^{(k-n+1)} + (f'(x)+g'(x)) (P_{n}^{(k-n+1)})^2 + 2 (f(x)+g(x))
P_{n}^{(k-n+1)} P_{n}^{(k-n+2)}\\
&=P_{n}^{(k-n+1)} \left(-P_{n-1}^{(k-n+1)} + g'(x) P_{n}^{(k-n+1)} + 2
(f(x)+g(x))P_{n}^{(k-n+2)} \right) + f'(x) (P_{n}^{(k-n+1)})^2.
\end{align*}
 It follows from the recurrence relation~\eqref{eq:Prec}
and Lemma~\ref{lem:APP} that
$$
-P_{n-1}^{(k-n+1)} = 2(2n+1)P_{n}^{(k-n)}-P_{n+1}^{(k-n+1)} = 2(2n+1)P_{n}^{(k-n)}-2(k+1)
P_{n}^{(k-n)} - (2x-1) P_{n}^{(k-n+1)}.
$$
 Thus,
\begin{multline*}
B'_{n,k}(x) = P_{n}^{(k-n+1)} \left(2(2n-k)P_{n}^{(k-n)}+ g'(x) P_{n}^{(k-n+1)} + 2
(f(x)+g(x))P_{n}^{(k-n+2)} \right) +\\+
f'(x) (P_{n}^{(k-n+1)})^2+(1-2x)(P_{n}^{(k-n+1)})^2.
\end{multline*}
 We choose
$f(x)$
and
$g(x)$
so that the first summand
is equal to zero.
 Taking into account the differential equation~\eqref{eq:Pdifur}
for
$m=n-k$,
we obtain the system
\begin{equation*}
\begin{cases}
f(x)+g(x) = \dfrac{(x-x^2)}{k+1}, \\
g'(x) = (2x-1)\dfrac{2(n-k-1)}{k+1}.
\end{cases}
\end{equation*}
 Therefore,
$$f'(x) = (1-2x) \dfrac{2(n-k)}{k+1}.$$
 Hence
$$
B'_{n,k}(x) = (1-2x) \dfrac{2(n-k)}{k+1} (P_{n}^{(k-n+1)})^2+(1-2x)(P_{n}^{(k-n+1)})^2 =
(1-2x)\dfrac{2n-k+1}{k+1}(P_{n}^{(k-n+1)})^2.
$$
 Thus, the function
$B_{n,k}(x)$
does not decrease
on the closed interval
$[0,{1}/{2}]$
and
does not increase
on the closed interval
$[{1}/{2},1]$,
while the function
$P_{n}^{(k-n+1)}$
is zero
at the maximum points of
$A^2_{n,k}(x)$,
whence we obtain the assertion of the theorem.
\end{proof}

\begin{picture}(400,150)
\put(0,0){
\begin{picture}(200,150)
\put(10,30){\includegraphics[scale=0.25]{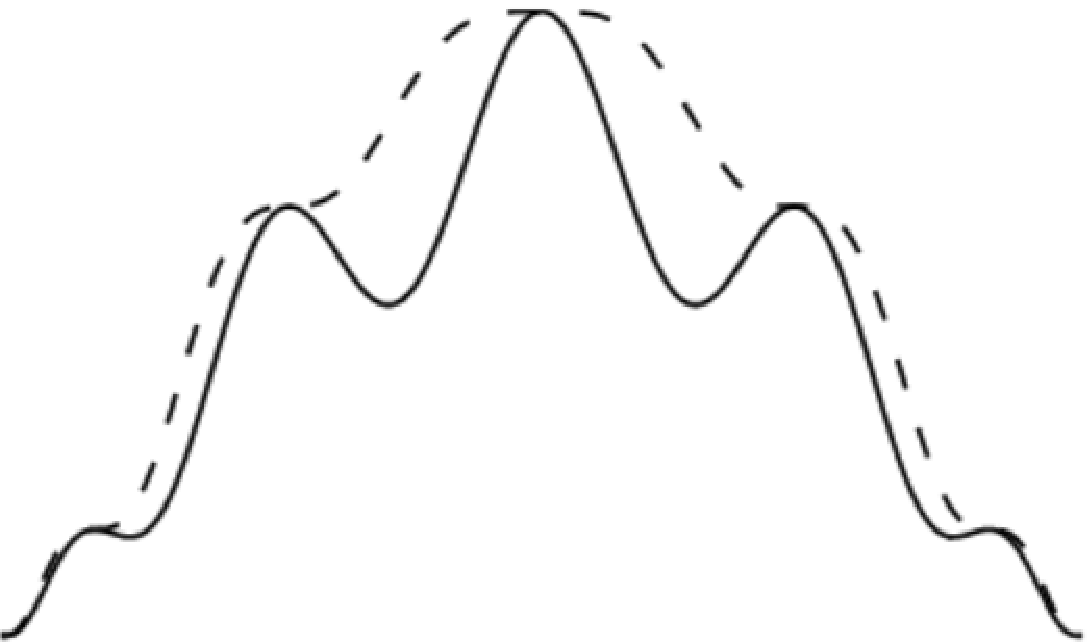}}
\put(0,30){\vector(1,0){160}}\put(155,22){$x$}
\put(10,10){\vector(0,1){110}}\put(2,109){$y$}
\put(5,20){$0$}
\put(135,20){$1$}
\multiput(74.5,28)(0,7){11}{\line(0,1){3}}
\put(68,20){$1/2$}
\put(30,05){Graphs of $A^2_{6,4}$ and $B_{6,4}$.}
\multiput(110,115)(7,0){3}{\line(1,0){3}}
\put(130,115){${}_{B_{6,4}}$}
\put(110,105){\line(1,0){20}}
\put(130,105){${}_{A^2_{6,4}}$}
\end{picture}
}
\put(280,0){
\begin{picture}(200,150)
\put(9,30){\includegraphics[scale=0.25]{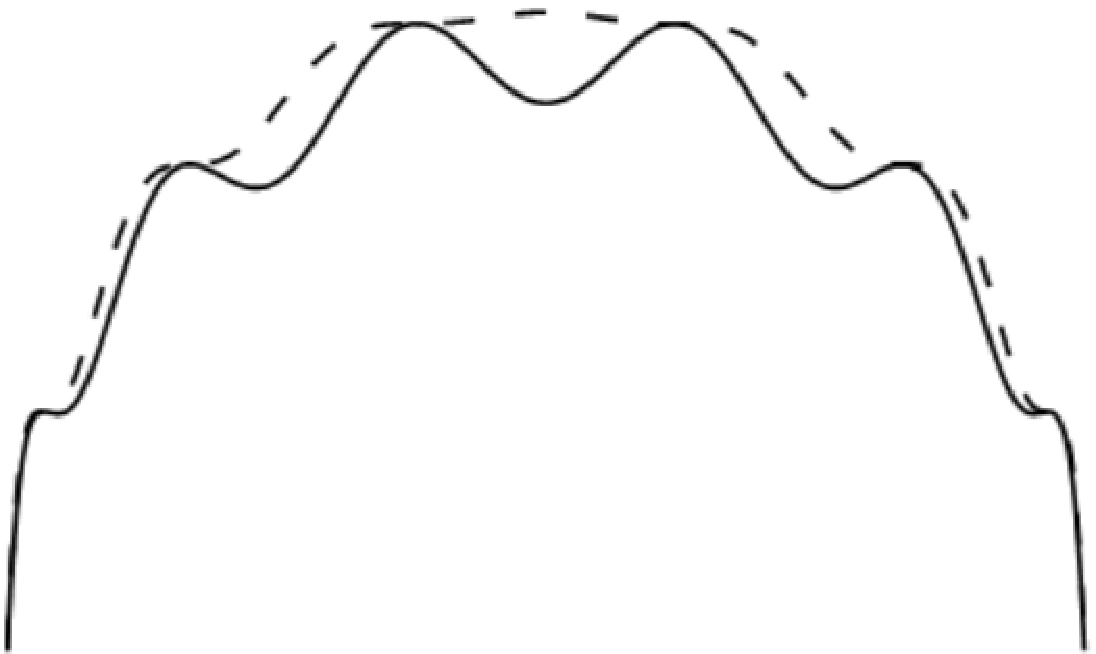}}
\put(0,30){\vector(1,0){160}}\put(155,22){$x$}
\put(10,10){\vector(0,1){110}}\put(2,109){$y$}
\put(5,20){$0$}
\put(135,20){$1$}
\multiput(74.5,28)(0,7){10}{\line(0,1){3}}
\put(68,20){$1/2$}
\multiput(110,115)(7,0){3}{\line(1,0){3}}
\put(130,115){${}_{B_{6,5}}$}
\put(110,105){\line(1,0){20}}
\put(130,105){${}_{A^2_{6,5}}$}
\put(30,05){Graphs of $A^2_{6,5}$ and $B_{6,5}$.}
\end{picture}
}
\end{picture}

\bigskip

 It follows from Lemma~\ref{lem:APP} that
the zeros of the polynomials
$P_{n-1}^{(k-n+1)}$
and
$P_{n}^{(k-n+1)}$
are
exactly
the extremal points of
$A^2_{n,k}(x)$.
 It follows from Theorem~\ref{thm:derAPP}
that the zeros of
$P_{n-1}^{(k-n+1)}$
are its minimum points
and the zeros of $P_{n}^{(k-n+1)}$
are the maximum points of $A^2_{n,k}(x)$.

 The properties of the maxima
and the minima of the quantities
$A^2_{n,k}(x)$
are similar to the properties of the values of~$|y(t)|$
at the extremal points of the solution~$y(t)$
of the equation
$-y''+p(t)y=0$
under the nonnegativity or nonpositivity condition for~$p'$
(the Sonin--P\'olya theorem; see~\cite[Sec.~19]{Trik},~\cite[6.2.2]{Adr}).
 The difference is in that, in the Sonin--P\'olya theorem,
the monotone sequence is formed by the absolute values
at all the extremal points of the function~$y$.

\begin{rim}
 The values of the functions
$A^2_{n,k}$
possess similar properties
at the minimum points.
 It follows from
Lemmas~\ref{lem:APP},~\ref{lem:zeroes}
that the zeros of the polynomial
$P_{n}^{(k-n+1)}$
are
the maximum points of the function
$A^2_{n,k}$
and,
simultaneously, the minimum points of the function
$A^2_{n+1,k+1}$.
\end{rim}

\section{The constants $\Lambda^2_{n,k}$. The general case of even~$k$}
\label{sub:Anchet}

 In this section,
we obtain expressions for the exact embedding constants
$\Lambda^2_{n,k}$
on the closed interval
$[-1;1]$
in the case of any even
$k\in[0;n-1]$.
 In this section,
the $P_n$ denote the Legendre polynomials
on the closed interval
$[-1;1]$.

 In~\cite{Kalyab}, the following formulas were obtained:
$$
\Lambda^2_{n,0}=A^2_{n,0}(0)=\dfrac{1}{2^{2n-1}((n-1)!)^2(2n-1)},\qquad
\Lambda^2_{n,2}=A^2_{n,2}(0)=\dfrac{1}{2^{2n-3}((n-2)!)^2(2n-5)}.
$$

\begin{thm}
 The exact values of the embedding constants
on the closed interval
$[-1;1]$
for
$k=2l$, $l=0,1,\ldots $
are of the form
$$
\Lambda^2_{n,k}=A^2_{n,k}(0)=\dfrac{((k-1)!!)^2}{2^{2n-k-1}((n-(k/2)-1)!)^2(2n-2k-1)}.
$$
\end{thm}

\begin{proof}
 Let us carry out the proof by the method of mathematical induction.
 The induction base was verified in~\cite{Kalyab}.

 It was conjectured in~\cite{Kalyab},
and proved in~\cite{NazM} that, for even~$k$,
the function
$A^2_{n,k}(x)$
has a local maximum
on the closed interval
$[-1;1]$
for
$x=0$.
 It follows from Theorem~\ref{thm:derAPP}
that, for even~$k$,
the point
$x=0$
is a global maximum point for the function
$A^2_{n,k}$.

 For the closed interval
$[-1;1]$,
formula~\eqref{eq:PA}
is written as follows
(see~\cite{Kalyab}):
$$
A^2_{n,k}(x)=A^2_{n-1,k-1}(x)-\left(n-\frac12\right)\left(P_{n-1}^{(k-n)}(x)\right)^2.
$$
 Therefore,
$$
A^2_{n-1,k-1}(x)=A^2_{n-2,k-2}(x)-\left(n-\frac32\right)\left(P_{n-2}^{(k-n)}(x)\right)^2,
$$
whence we have
$$
A^2_{n,k}(x)=A^2_{n-2,k-2}(x)-\bigl(P_{n-2}^{(k-n)}(x)\bigr)^2\left(n-\frac32\right)-\bigl
(P_{n-1}^{(k-n)}(x)\bigr)^2\left(n-\frac12\right).
$$

 In particular, this formula is also valid
for
$x=0$.
  The antiderivative of order $n-k$
of the Legendre polynomial
$P_{n-1}$
is of the form
$$
P_{n-1}^{(k-n)}(x)=\dfrac{1}{2^{n-1}}\dfrac{1}{(n-1)!}\left((x^2-1)^{n-1}\right)^{(k-1)}.
$$

 The polynomial
$(x^2-1)^{n-1}$
contains only even powers of~$x$;
if
$k=2l$
is an even number,
then
$\left((x^2-1)^{n-1}\right)^{(k-1)}$
only contains odd powers of~$x$;
therefore,
$\bigl(P_{n-1}^{(k-n)}(0)\bigr)^2=0$.

 For the polynomial
$P_{n-2}$,
the situation is different:
$$
P_{n-2}^{(k-n)}(x)=\dfrac{1}{2^{n-2}}\dfrac{1}{(n-2)!}\left((x^2-1)^{n-2}\right)^{(k-2)}.
$$
 If
$k=2l$
is an even number,
then $k-2$
is also an even number.
 Therefore,%
\footnote{Here and elsewhere,
$C_n^m$
denotes the binomial coefficient
$\binom nm$.}
$$
(x^2-1)^{n-2}=\sum_{j=0}^{n-2}C_{n-2}^j x^{2j}\cdot (-1)^{n-2-j}.
$$

 For
$x=0$,
the nonzero contribution to the expression
$$
\left(\sum_{j=0}^{n-2}C_{n-2}^j x^{2j}\cdot (-1)^{n-2-j}\right)^{(k-2)}
$$
is only made by the summand
containing $x$
with exponent
$2j=k-2$.
 This contribution is
$$
C_{n-2}^{{(k-2)}/{2}}(k-2)!(-1)^{n-2-j}.
$$

 As a result, we obtain
$$
\bigl(P_{n-2}^{(k-n)}(0)\bigr)^2\left(n-\frac32\right)=\dfrac{1}{2^{2n-4}}\dfrac{1}{((n-2)
!)^2}
\left(C_{n-2}^{{(k-2)}/{2}}(k-2)!\right)^2\left(n-\frac32\right).
$$
 This expression can be written as
$$
\dfrac{1}{2^{2n-4}}\dfrac{1}{((n-2)!)^2}
\dfrac{((n-2)!)^2((k-2)!)^2\left(n-\frac32\right)}{((\frac{k-2}{2})!)^2((n-2-k/2+1)!)^2}=
\dfrac{((k-2)!)^2(2n-3)}{2^{2n-3}((\frac{k-2}{2})!)^2((n-k/2-1)!)^2}.
$$

 Since
$$
\left(\left(\frac{k-2}{2}\right)!\right)^2=\dfrac{((k-2)!!)^2}{2^{k-2}},
$$
we have
\begin{equation}
\label{eq:P}
\bigl(P_{n-2}^{(k-n)}(0)\bigr)^2\left(n-\frac32\right)=\dfrac{((k-3)!!)^2(2n-3)}{2^{2n-k-1
}((n-k/2-1)!)^2}.
\end{equation}

 By the induction assumption,
for an even~$k$,
\begin{equation}
\label{eq:A}
A_{n-2,k-2}^2(0)=\dfrac{((k-3)!!)^2}{2^{2n-k-3}((n-2-k/2)!)^2(2n-2k-1)}
\end{equation}
 It remains to find the difference of~\eqref{eq:P}
and~\eqref{eq:A}:
$$
A_{n,k}^2(0)=\dfrac{((k-3)!!)^2}{2^{2n-k-3}((n-2-k/2)!)^2(2n-2k-1)}-\dfrac{((k-3)!!)^2(2n-
3)}{2^{2n-k-1}((n-k/2-1)!)^2}.
$$
 In this difference,
taking the common multiplier
$$
\dfrac{((k-3)!!)^2}{2^{2n-k-1}((n-k/2-1)!)^2(2n-2k-1)}
$$
out of the bracket,
we obtain the remaining difference
$$
(2n-2-k)^2-(2n-3)(2n-2k-1)=(k-1)^2.
$$

 The application of the method of mathematical induction
concludes the proof.
\end{proof}

\section{Embedding constants on the closed interval $[0;1]$}
 The following formula
relating the functions
$A^2_{n,k}$
defined on the closed intervals
$[0;1]$
and
$[-1;1]$:
$$
A^2_{n,k,[-1;1]}(2x-1)=2^{2n-2k-1}A^2_{n,k,[0;1]}(x).
$$
was established in~\cite{GarmSh}
on the basis of the properties of the spectral problem
\begin{gather}
\label{eq:krzad}
(-1)^ny^{(2n)}=\lambda (-1)^k y^{(k)}(a)\delta^{(k)}(x-a),\\
y^{(j)}(0)=y^{(j)}(1)=0,\quad j=0,1,\ldots,n-1.
\end{gather}

 Using the relation between $A^2_{n,k,[-1;1]}(2x-1)$ and $A^2_{n,k,[0;1]}(x)$,
we see that,
on the closed interval
$[0;1]$,
the exact value of the embedding constant
is of the form
$$
\Lambda^2_{n,k,[0;1]}:=A^2_{n,k,[0;1]}(1/2)=\dfrac{((k-1)!!)^2}{2^{4n-3k-2}((n-(k/2)-1)!)^
2(2n-2k-1)}.
$$

 This recalculation formula
can also be obtained
directly from the definition of~$A^2_{n,k}$~\eqref{eq:ekstr}
on the closed intervals
$[0;1]$
and
$[-1;1]$,
respectively.

\section{On the symmetry of the extremal function}

 The paper~\cite{GarmSh}
also contains the following explicit expressions for the extremal splines:
\begin{equation}
\label{eq:g}
g_{n,k}(x)=\left\{\begin{aligned}
&\dfrac{(-1)^{n-k-1}}{(2n-k-1)!}(1-a)^{n-k}x^n h_{n,k}(1-x,1-a),\quad x\in [0;a]\\
&\dfrac{(-1)^{n-1}}{(2n-k-1)!}a^{n-k}(1-x)^{n} h_{n,k}(x,a), \quad x\in [a;1],
\end{aligned}\right.
\end{equation}
where the polynomials
$h_{n,k}(x,a)$
are defined by the formula
$$
h_{n,k}(x,a)=\sum_{l=0}^{n-1}(-1)^{n-1-l}C_{2n-1-k}^{n-1-l}x^{n-1-l}a^l\sum_{m=0}^{l}C_{n-
1+m}^{m}x^{m}.
$$

 It is seen from formula~\eqref{eq:g} that
the extremal spline is symmetric
if and only if
$a=\frac12$;
thus, for even~$k$,
the extremal spline possesses the symmetry property:
$$
g_{n,k}(x,1/2)=g_{n,k}(1-x,1/2),
$$
while, for odd~$k$,
the extremal spline is not symmetric.


\end{document}